\documentclass[12pt,reqno]{amsart}

\usepackage{preamble_paper}
\usepackage{preamble_general}
\usepackage{preamble_math}
\usepackage{mathtools}

\renewcommand{\P}{\mathbf{P}}

\begin{document}

\title{The inverse problem for the Steiner--Wiener index of trees}

\author{Adrian Beker}
\address{University of Zagreb, Faculty of Science, Department of Mathematics, Zagreb, Croatia}
\email{Adrian.Beker@math.hr}

\author{Rudi Mrazovi\'{c}}
\address{University of Zagreb, Faculty of Science, Department of Mathematics, Zagreb, Croatia}
\email{Rudi.Mrazovic@math.hr}

\begin{abstract}
    For a connected graph $G$ and a set $S\subset V(G)$,
    the Steiner distance $d_G(S)$ is the minimum number of edges in a connected subgraph of $G$ containing $S$.
    The Steiner--Wiener $k$ index is defined by
    \[
        \SW_k(G)
        =
        \sum_{\substack{S\subset V(G)\\ |S|=k}} d_G(S).
    \]
    We study the inverse problem for this invariant restricted to trees: 
    for fixed $k$,
    which positive integers occur as $\SW_k(T)$ for a finite tree $T$?
    We prove that all sufficiently large positive integers occur as $\SW_k(T)$ for some finite tree $T$ if and only if $k$ is even.
    For odd $k$, 
    we further show that the set of attainable values has asymptotic density of order $k^{-\delta}(\log k)^{-3/2}$, where $\delta$ is the Erd\H{o}s--Tenenbaum--Ford constant.
\end{abstract}

\maketitle

\section{Introduction}

The Wiener index is a classical distance-based graph invariant in chemical graph theory.
For a connected graph $G$, it is defined by
\[
    \W(G)
    =
    \sum_{\{u,v\}\subset V(G)}
        d_G(u,v),
\]
where $d_G(u,v)$ is the usual graph distance. 
This invariant was introduced by Wiener in his study of boiling points of paraffins \cite{Wiener},
and has since become a central example of a topological index. 
The case when $G$ is a tree has played an especially important role, 
both because many molecular graphs of interest are naturally modeled by trees and because trees admit particularly effective combinatorial formulas;
see, for example, the survey of Dobrynin, Entringer, and Gutman \cite{DobryninEntringerGutman}.

The Steiner--Wiener $k$ index is a multi-vertex analogue of the Wiener index.
If $G$ is connected and $S\subset V(G)$,
the \emph{Steiner distance} $d_G(S)$ is the minimum number of edges in a connected subgraph of $G$ containing every vertex of $S$. 
For $k\geq 2$, the \emph{Steiner--Wiener $k$ index} is defined by
\[
    \SW_k(G)
    =
    \sum_{\substack{S\subset V(G)\\ |S|=k}}
        d_G(S).
\]
When $k=2$, this is exactly the ordinary Wiener index. 
This invariant was introduced by Li, Mao and Gutman \cite{LiMaoGutman}, 
building on earlier work on average Steiner distance \cite{DankelmannOellermannSwart}.

In this paper, we study the corresponding inverse problem, first proposed in \cite{LiMaoGutmanInverse} (see also \cite{ZhangGentryWangJinZhang}). For fixed $k$, let
\begin{equation}
    \label{eq:all-sw-for-k}
    \cT_k
    \vcentcolon=
    \{\SW_k(T):T\text{ is a finite tree}\}.
\end{equation}
Which positive integers belong to $\cT_k$? 
When $k=2$, this is the classical inverse problem for Wiener indices of trees.
A result of Wagner \cite{Wagner} and, 
independently, Wang and Yu \cite{WangYu}, 
shows that all but finitely many positive integers are Wiener indices of trees.

For general connected graphs, 
Bernert and Shaw \cite{BernertShaw} recently obtained the analogous cofiniteness result:
for every fixed $k\geq 2$, 
every sufficiently large positive integer occurs as $\SW_k(G)$ for some connected graph $G$.
They also observed that,
from the viewpoint of chemical applications,
it is natural to impose structural restrictions such as being a tree,
and that the inverse problem becomes substantially more difficult in this setting. Our first main result addresses the case of trees by identifying all $k$ for which the cofiniteness phenomenon persists.

\begin{theorem}
    \label{t:cofinite}
    Let $k\geq 2$. Then
    \[
        \cT_k
        =
        \{\SW_k(T):T\text{ is a finite tree}\}
    \]
    is cofinite if and only if $k$ is even.
\end{theorem}

Theorem~\ref{t:cofinite} provides an essentially complete description of the possible values of the Steiner--Wiener $k$ index if $k$ is even, but leaves open the corresponding question when $k$ is odd. In this case, our second result shows that $\cT_k$ is closely related to a set of integers with a localised divisor, thereby revealing an unexpected link with the anatomy of integers. In turn, this connection allows us to give a precise estimate for the size of $\cT_k$.

Given a set $E\subseteq\N$, we write $d(E)$ for its asymptotic density (when it exists).
Given $0<\alpha<\beta<1$, we write
\begin{equation}\label{eq:localised-density}
    h(\alpha,\beta)
    \vcentcolon=
    d\bigl(
        \{n\in\N : n \text{ has a divisor in }(n^\alpha,n^\beta]\}
    \bigr)
\end{equation}
for the localised-divisor density\footnote{It is a standard fact that this density admits the alternative expression
\[
    h(\alpha,\beta) = \lim_{x\to\infty}\frac{H(x,x^\alpha,x^\beta)}{x},
\]
where $H(x,y,z)$ is the number of $n \leq x$ with a divisor in $(y,z]$; for a proof in the case $\beta < \frac{1}{2}$, see Remark~\ref{r:density}.}, whose existence was proved by Tenenbaum \cite{MR574122}.
Finally, we let
\begin{equation}
    \label{eq:ETFconstant}
    \delta 
    \vcentcolon=
    1 - \frac{1+\log\log2}{\log2}
    \approx 
    0.086
\end{equation}
be the Erd\H{o}s--Tenenbaum--Ford constant.

\begin{theorem}
    \label{t:density}
    For every odd $k\geq 3$, we have
    \[
        d(\cT_k\triangle\cD_k) = 0, \quad \text{where} \quad \cD_k 
        \vcentcolon=
        \bigl\{n \in \N : \text{$n$ has a divisor in }(n^{1/(k+1)}, n^{1/k}]\bigr\}.
    \]
    In particular, the asymptotic density of $\cT_k$ exists and satisfies
    \[
        d(\cT_k)
        = 
        h\Bigl(\frac1{k+1},\frac1k\Bigr) \asymp k^{-\delta}(\log k)^{-3/2}.
    \]
\end{theorem}

\subsection*{Overview of the arguments.}
The arguments in the paper naturally split into two parts.
The first part, culminating in the proof of Theorem~\ref{t:cofinite}, is constructive and mainly combinatorial:
it is based on explicit constructions.
The second part, leading to Theorem~\ref{t:density}, is probabilistic and more analytic in flavour. In addition to the connection to additive number theory, already observed in \cite{BernertShaw}, this part showcases a link with multiplicative number theory.

We now give a brief overview of the arguments. For a tree $T$ on $n$ vertices,
deleting an edge $e$ separates $T$ into two components, of sizes $m_e$ and $n-m_e$ say. 
The edge $e$ belongs to the smallest subtree containing a $k$-set $S \subseteq V(T)$ exactly when $S$ meets both components. 
Hence, by double counting,
\begin{equation}
    \label{eq:double-counting}
    \SW_k(T)
    =
    \sum_{e\in E(T)}
    \Bigl(
        \binom{n}{k}-\binom{m_e}{k}-\binom{n-m_e}{k}
    \Bigr).
\end{equation}
For odd $k$, after multiplying the summands by $k!$, 
one obtains a polynomial whose coefficients are multiples of $n-k+1$,
which gives a divisibility obstruction.

The even-$k$ case is constructive. 
We work with caterpillars, which are trees possessing a special structure that makes the analysis of the Steiner--Wiener index tractable.\footnote{We note that caterpillars already feature in the proof of the $k = 2$ case \cite{WangYu}.}
For each $A\subseteq [n-1]$, 
we build an $n$-vertex caterpillar $C_A^n$ whose components after removing spine edges have sizes precisely the elements of $A$ 
(see Section \ref{s:prelim} for details).
Then
\[
    \SW_k(C_A^n)
    =
    G_{n,k}+\sum_{a\in A} F_{n,k}(a),
\]
for some explicit polynomials $G_{n,k}$ and $F_{n,k}(a)$ (see \eqref{eq:GnFn}).
Thus the inverse problem is transformed into a subset-sum problem for the
polynomial sequence $F_{n,k}(1),\dots,F_{n,k}(n-1)$.

The key device in the analysis of the above-mentioned subset sums is a switch, i.e.\ 
a local modification of a subset of $[n-1]$.
When the two sets constituting a switch are chosen as the sign classes of a mixed finite difference, the subset sum changes by
\[
    \bigl|
        \Delta_{h_1}\dots\Delta_{h_r}F_{n,k}(s)
    \bigr|,
\]
a quantity we call the switch value.
Note that this is particularly amenable to analysis given that the $F_{n,k}$ are polynomial functions.
For even $k$, $k$-fold switches have constant values,
while lower-order switches give controlled values on the scales $n,n^2,\dots,n^k$.
By choosing many disjoint switches and using a simple interval-overlap argument for subset sums,
we obtain intervals of attainable switch-values of length $\Omega(n^{k+1})$. 
These intervals persist coherently as $n$ varies, 
and their lengths dominate the $O(n^k)$ change in the corresponding base subset-sum values.
Therefore the intervals for consecutive even $n$ overlap,
which proves cofiniteness of $\cT_k$.

To prove Theorem~\ref{t:density}, 
we show that almost all numbers satisfying the obvious divisibility- and size-obstructions are realisable as Steiner--Wiener indices. Let $T$ be an $n$-vertex caterpillar. The argument behind Theorem~\ref{t:cofinite} yields an interval of values of $\SW_k(T)$ of close to maximum length. However, it does not provide sufficient control over the precise location of this interval. Hence, we employ a different approach, which is probabilistic in nature. The idea is to consider a random subset of $[n-1]$ obtained by sampling each element independently with some probability $p \in (0,1)$. We then prove a local central limit theorem for the sum of the values of $F_{n,k}$ on this subset. Dropping the probabilistic language, this allows us to count subset-sum representations of each value in a sufficiently long interval, and in particular to show that this number is non-zero. The local limit theorem applies to a more general class of polynomials than those considered here, and hence might be of independent interest.

As usual, the proof of the local central limit theorem proceeds by Fourier inversion on the appropriate lattice, call it $g_{n,k}\Z$.
On a small neighborhood of the origin,
a result of Giuliano and Weber \cite{GiulianoWeber} gives the expected Gaussian main term.
The main difficulty is to control the remaining frequencies. 
The product formula for the characteristic function gives exponential decay unless most of the normalised values $F_{n,k}(j)/g_{n,k}$ lie in a small rank-one Bohr set.
In the latter case,
an additive-combinatorial argument shows that a bounded iterated sumset of these values contains a long arithmetic progression.
This forces the frequency to be very close to an integer; 
the normalization by $g_{n,k}$ then rules out all nonzero lattice obstructions. 
This yields the required minor-arc estimate and completes the proof of the local central limit theorem.

It follows that $\cT_k$ is sandwiched between two unions, over $n$, 
of intervals of multiples of $g_{n,k}$, 
with the two unions differing only in boundary layers of suitably small width.
Ford's estimates \cite{Ford} for integers having a divisor in a prescribed interval then allow us to relate $\cT_k$ to the localised-divisor set
\[
    \bigl\{
        N\in\N :
        N\text{ has a divisor in }(N^{1/(k+1)},N^{1/k}]
    \bigr\}.
\]
By Tenenbaum's theorem on localised divisors, the density of this set exists and Ford's estimates allow us to determine its order of magnitude.

\subsection*{Organisation of the paper.}
In Section \ref{s:prelim} we collect the basic identities for the Steiner--Wiener index of trees, and introduce caterpillars together with the corresponding switch formalism. Sections~\ref{s:odd-k} and~\ref{s:even-k} are devoted to the proofs of the odd- and even-$k$ directions of Theorem~\ref{t:cofinite} respectively, and constitute the combinatorial part of the paper. Sections~\ref{s:lclt} and~\ref{s:density} form the probabilistic/analytic part.
Section \ref{s:lclt} establishes the local central limit theorem and its consequence for subset sums of polynomial values.
Section \ref{s:density} combines that result with Ford's divisor estimates to prove Theorem~\ref{t:density}; 
it also gives an alternative proof of the even-$k$ direction of Theorem \ref{t:cofinite}.
Section~\ref{s:further} provides some concluding remarks, as well as a discussion of open problems and potential directions for future research.
Appendix \ref{s:gcd} contains a proof of an exact formula for $g_{n,k}$.

\subsection*{Notation.}
Throughout, $\N$ denotes the set of positive integers and we write $[m]=\{1,\ldots,m\}$.
Intervals are interpreted according to the ambient set: 
as intervals of real numbers in analytic statements and as discrete intervals when the ambient set is $\Z$ or $\N$.
We write $\lVert x\rVert_{\R/\Z}$ for the distance from $x$ to the nearest integer.
The notations $A=O(B)$ and $A\ll B$ both mean that $|A|\leq CB$ for some constant $C>0$,
while $A=\Omega(B)$ means $A\geq cB$ for some constant $c>0$ when $A,B$ are non-negative. 
We write $A\asymp B$ if $A\ll B$ and $B\ll A$.
Dependence of implicit constants on parameters is indicated by subscripts. Logarithms are to base $e$ unless otherwise indicated.


\section{Preliminaries}
\label{s:prelim}

Let $T$ be a tree on $n\geq k$ vertices. For an edge $e\in E(T)$, let
$m_e$ and $n-m_e$ be the sizes of the two components of $T-e$. For a $k$-subset $S\subseteq V(T)$, the edge $e$ is present in the unique smallest subtree containing $S$ 
if and only if $S$ meets both components of $T-e$. Therefore, by double
counting such pairs $(e,S)$, we have
\[
    \SW_k(T)
    =
    \sum_{e\in E(T)}
    \Bigl(
        \binom{n}{k}
        -
        \binom{m_e}{k}
        -
        \binom{n-m_e}{k}
    \Bigr).
\]
It will be more convenient to work with a slightly different expression, namely
\begin{equation}
    \label{eq:sw-trees}
    \SW_k(T) = G_{n,k} + \sum_{e\in E(T)}F_{n,k}(m_e),
\end{equation}
where
\begin{equation}
    \label{eq:GnFn}
    G_{n,k}
    \vcentcolon=
    (n-1) \binom{n-1}{k-1},
    \qquad
    F_{n,k}(x)
    \vcentcolon=
    \binom{n-1}{k} - \binom{x}{k} - \binom{n-x}{k}.
\end{equation}

To show that certain values are realisable as the Steiner--Wiener index, we will work with caterpillars, that is, trees whose non-leaf vertices lie on a
single path, called the \emph{spine}. Caterpillars with $n$ vertices are in natural $1$-$1$ correspondence with subsets of $[n-1]$; the following definition makes this precise.

\begin{definition}
    Let $n\geq 2$, and let
    \[
        A=\{a_1<\dots<a_m\}\subset [n-1].
    \]
    be an arbitrary subset. Additionally, let $a_0= 0$ and $a_{m+1} = n$. 
    We define $C_A^n$ to be the caterpillar with spine vertices
    $v_0,v_1,\dots,v_m$ and $a_{i+1}-a_i-1$ leaves attached to $v_i$ for each $i \in \{0,1,\ldots,m\}$.
\end{definition}

Note that for the caterpillar as in the above definition,
removing the spine edge $v_{i-1}v_i$ splits the tree into components of sizes $a_i$ and $n-a_i$.
Removing any non-spine edge yields components of sizes $1$ and $n-1$. 
Therefore, specialising \eqref{eq:sw-trees} to caterpillars gives
\begin{equation}
    \label{eq:sw-caterpillar}
    \SW_k(C_A^n)
    =
    G_{n,k} + \sum_{a\in A}F_{n,k}(a).
\end{equation}

The proof of Theorem~\ref{t:cofinite} for even $k$ is based on controlled modifications of the set $A$.

\begin{definition}
    Fix $n$. A \emph{switch} is an ordered pair $\cP=(B^-,B^+)$ of disjoint subsets of $[n-1]$.
    We call $B^-$ the \emph{negative set},
    $B^+$ the \emph{positive set}, 
    and $B^-\cup B^+$ the \emph{support} of $\cP$.

    The switch $\cP$ may be applied to a set $A\subset[n-1]$ if $B^- \subset A$ and $B^+\cap A = \emptyset$.
    Applying $\cP$ replaces $A$ by $(A \setminus B^-) \cup B^+$.

    The \emph{value} of the switch $\cP$ is defined to be
    \[
        v_{n,k}(\cP)
        \vcentcolon=
        \sum_{a\in B^+} F_{n,k}(a) - \sum_{a\in B^-} F_{n,k}(a).
    \]
    This is the amount by which $\SW_k(C_A^n)$ changes when $\cP$ is applied to $A$.
\end{definition}

We shall only use switches with positive value.
If a switch has negative value, we interchange its positive and negative sets.
To preserve coherence as $n$ varies (see the beginning of Section~\ref{s:even-k}),
it will be important that, within each family of switches used below,
the need for this interchange is independent of $n$ for all sufficiently large $n$.

A collection of switches $\fP$ is called \emph{compatible} if their supports are pairwise disjoint.
For a compatible collection $\fP$, we define
\[
    A(\fP)
    \vcentcolon=
    \bigcup_{\cP\in\fP} B^-(\cP)
    \qquad \text{and} \qquad
    V_{n,k}(\fP)
    \vcentcolon=
    \Bigl\{
        \sum_{\cP\in\fQ} v_{n,k}(\cP)
        :
        \fQ \subset \fP
    \Bigr\}.
\]
Then every value in
\[
    \SW_k(C_{A(\fP)}^n)+V_{n,k}(\fP)
\]
is realised as $\SW_k(C_A^n)$ for some $A\subset[n-1]$, 
since every subcollection of switches $\fQ \subset \fP$ may be applied to $A(\fP)$. We will repeatedly exploit this through the following simple observation. If $\fP'$ is a compatible collection obtained from $\fP$ by adding a single switch $\cP$, then
\[
    V_{n,k}(\fP') = V_{n,k}(\fP) \cup \bigl(V_{n,k}(\fP) + v_{n,k}(\cP)\bigr).
\]
In particular, if $V_{n,k}(\fP)$ contains an interval $I \subseteq \Z$ of length $|I| \geq v_{n,k}(\cP)$, then $V_{n,k}(\fP')$ contains $I' \vcentcolon= I \cup (I + v_{n,k}(\cP))$, an interval of length $|I| + v_{n,k}(\cP)$ and with the same left endpoint as $I$.

We next introduce finite-difference switches. 
Let $r\geq 1$ and $s, h_1, \dots, h_r \geq 1$, and suppose that the $2^r$ numbers
\begin{equation}
    \label{eq:fin-dif-switch-supp}
    s + \eta_1 h_1 + \dots + \eta_r h_r,
    \qquad 
    \eta = (\eta_1, \dots, \eta_r) \in \{0,1\}^r,
\end{equation}
are distinct elements of $[n-1]$.
The \emph{$r$-fold switch with base $s$ and steps $h_1,\dots,h_r$} is the switch whose positive and negative sets consist of the elements $s + \eta_1 h_1 + \dots + \eta_r h_r$ for which $|\eta|_1 \vcentcolon= \sum_{i=1}^{r}\eta_i$ is even and odd.
Note that, after possibly interchanging the positive and negative sets, the value of this switch is
\begin{equation}
    \label{eq:fin-dif-switch-val-1}
    \Bigl| 
        \sum_{\eta\in\{0,1\}^r}
            (-1)^{r-|\eta|_1} F_{n,k}(s + \eta_1h_1 + \dots + \eta_rh_r) 
    \Bigr|
    =
    \bigl|
        \Delta_{h_1} \dots \Delta_{h_r} F_{n,k}(s) 
    \bigr|,
\end{equation}
where we define the forward difference operator $\Delta_h f(x)\vcentcolon=f(x+h)-f(x)$.
For later use, we retain the orientation in which the even-parity elements form the positive set and write
\[
    \tilde{v}_{n,k}(s;\mathbf{h})
    \vcentcolon=
    \sum_{\eta\in\{0,1\}^r}
        (-1)^{|\eta|_1}
        F_{n,k}(s+\eta_1h_1+\dots+\eta_rh_r).
\]
Thus, $\tilde{v}_{n,k}(s;\mathbf{h}) = (-1)^r \Delta_{h_1}\dots\Delta_{h_r}F_{n,k}(s)$, and the positive and negative sets are interchanged precisely when $\tilde{v}_{n,k}(s;\mathbf{h})<0$.

We end this section with a simple lemma which collects some properties of the polynomials $F_{n,k}$ depending on the parity of $k$. Here and in what follows, we let
\begin{equation}
    \label{eq:gcd}
    g_{n,k} \vcentcolon= \gcd\bigl\{F_{n,k}(x) : x \in \Z\bigr\}.
\end{equation}

\begin{lemma}
    \label{l:even-vs-odd}
    We have the following properties of $F_{n,k}(x)$ as a polynomial in $x$:
    \begin{enumerate}
        \item Assume $k$ is even. Then $F_{n,k}$ has degree $k$, leading coefficient $-2/k!$ and satisfies
        \begin{equation}
            \label{eq:lin-comb-even}
            \Bigl|\sum_{\ell=0}^{k}(-1)^{k-\ell}\binom{k}{\ell}F_{n,k}(x+\ell)\Bigr| = 2,
        \end{equation}
        so in particular $g_{n,k} \in \{1,2\}$.
        \item Assume $k$ is odd. Then $F_{n,k}$ has degree $k-1$, leading coefficient $-\frac{n-k+1}{(k-1)!}$ and satisfies
        \begin{equation}
            \label{eq:lin-comb-odd}
            \Bigl|\sum_{\ell=0}^{k-1}(-1)^{k-1-\ell}\binom{k-1}{\ell}F_{n,k}(x+\ell)\Bigr| = n-k+1,
        \end{equation}
        so in particular $g_{n,k}$ divides $n-k+1$. Furthermore, all coefficients of $k!F_{n,k}$ are divisible by $n-k+1$, so in particular $g_{n,k}$ is divisible by $\frac{n-k+1}{(n-k+1,k!)}$.
    \end{enumerate}
\end{lemma}

\begin{proof}
    The statements about the degree and leading coefficient of the polynomial $F_{n,k}$ follow by simple calculations. The identities~\eqref{eq:lin-comb-even} and~\eqref{eq:lin-comb-odd} then follow from~\eqref{eq:fin-dif-switch-val-1} with $r = \deg F_{n,k}$ and $h_1 = \ldots = h_r = 1$, and the fact that for a degree-$r$ polynomial $f$, $\Delta_1^rf$ equals $r!$ times its leading coefficient. Finally, to prove that the coefficients of $k!F_{n,k}$ are divisible by $n-k+1$, note that
    \[
        k!F_{n,k}(x) = (n-1)^{\underline{k}} - x^{\underline{k}} - (n-x)^{\underline{k}},
    \]
    where $z^{\underline{k}}$ denotes the falling factorial polynomial $z (z-1) \dots (z-k+1)$. Clearly $(n-1)^{\underline{k}}$ is divisible by $n-k+1$, and
    \[
        (n-x)^{\underline{k}} = (n-x)(n-x-1)\ldots(n-x-k+1)
    \]
    is congruent modulo $n-k+1$ to
    \[
        (k-1-x)(k-2-x)\ldots(-x) = (-1)^k(x-k+1)(x-k+2)\ldots x = -x^{\underline{k}}.
    \]
    This concludes the proof.
\end{proof}


\section{Proof of Theorem~\ref{t:cofinite} for odd \texorpdfstring{$k$}{k}}
\label{s:odd-k}

In this section we prove the odd-$k$ direction of Theorem~\ref{t:cofinite},
i.e.\ that for every odd $k\geq 3$, the set
\[
    \cT_k
    =
    \{\SW_k(T):T\text{ is a finite tree}\}
\]
is not cofinite.

Note that, for any tree $T$ on $n$ vertices, $\SW_k(T)$ is divisible by
\[
    h_{n,k} 
    \vcentcolon= 
    \frac{n-k+1}{\gcd(n-k+1,k!)},
\]
because by Lemma~\ref{l:even-vs-odd} the same is true for each summand in \eqref{eq:sw-trees}.
Obviously, $h_{n,k}>1$ for sufficiently large $n$. Moreover, $h_{n,k}\leq n-k+1\leq n$, whereas
\[
    \SW_k(T)
    \geq 
    (k-1)\binom{n}{k},
\]
since every $k$-set has Steiner distance at least $k-1$.
Thus, for sufficiently large $n$ we have $h_{n,k} < \SW_k(T)$ and hence $h_{n,k}$ is a proper divisor of $\SW_k(T)$. In particular, for a sufficiently large finite tree $T$, $\SW_k(T)$ cannot be prime, so there are infinitely many numbers not in $\cT_k$.

We note that the modulus $h_{n,k}$ suffices for this argument,
although we will later identify the \emph{right} divisor in Lemma~\ref{l:gcd}.


\section{Proof of Theorem~\ref{t:cofinite} for even \texorpdfstring{$k$}{k}}
\label{s:even-k}

Throughout this section, $k$ is fixed and even. In particular, we allow implicit constants to depend on $k$. In what follows, we say that the collections of switches $(\fQ_n)$, defined for sufficiently large even $n$, are \emph{coherent} if the following is satisfied for all such $n$: $\fQ_n$ is compatible and all switches in $\fQ_n$ have support of size $O(1)$; in addition, there is a pairing between $\fQ_n$ and $\fQ_{n+2}$ such that:
\begin{enumerate}
    \item all but $O(1)$ switches are paired;
    \item paired switches are ordered translates of one another by $O(1)$,
    i.e.\ if
    \[
        \cP=(B^-,B^+)\in\fQ_n
    \]
    is paired with $\cP'\in\fQ_{n+2}$, then
    \[
        \cP'=(B^-+t,B^++t)
    \]
    for some integer $t=O(1)$.
\end{enumerate}
We note that the bounded-size support condition is not really an issue since we will almost exclusively work with $r$-fold switches for $r \leq k$, which have support of size at most $2^k$.

We shall establish the following interval statement, from which Theorem~\ref{t:cofinite} follows easily. 

\begin{proposition}
    \label{p:interval-n-k-plus-1}
    There exist constants $\eps>0$ and $\ell\geq 1$ for which there are coherent collections of switches $(\fP_n)$ such that
    \[
        [\ell,\ell+\eps n^{k+1}]
        \subset
        V_{n,k}(\fP_n).
    \]
\end{proposition}

\begin{proof}[Proof of Theorem~\ref{t:cofinite} for even $k$]
    It is enough to prove cofiniteness using only caterpillars.
    We begin by applying Proposition~\ref{p:interval-n-k-plus-1}. For each sufficiently large even $n$, let $A_n=A(\fP_n)$ and $b_n=\SW_k(C_{A_n}^n)$.
    Since $[\ell, \ell+\eps n^{k+1}] \subset V_{n,k}(\fP_n)$,
    we have
    \[
        J_n
        \vcentcolon=
        [b_n+\ell,\, b_n+\ell+\eps n^{k+1}]
        \subset
        \{
            \SW_k(C_A^n)
            :
            A\subset[n-1]
        \}.
    \]

    We claim that $|b_{n+2}-b_n|=O(n^k)$.
    Indeed,
    \[
        |b_{n+2} - b_n|
        \leq
        |G_{n+2,k} - G_{n,k}|
        + 
        \Bigl|
            \sum_{a\in A_n} F_{n,k}(a)
            - \sum_{a\in A_{n+2}} F_{n+2,k}(a)
        \Bigr|.
    \]
    Obviously, $|G_{n+2,k}-G_{n,k}| = O(n^{k-1})$. Furthermore,
    the coherent choice of switch collections implies that all but $O(1)$ elements $a$ of $A_n$ are paired with all but $O(1)$ elements $a'$ of $A_{n+2}$, and the paired elements satisfy $a'=a+O(1)$.
    Since $F_{n,k}(a)$ is a polynomial of total degree $k$ in $n$ and $a$, 
    we have $F_{n+2,k}(a')-F_{n,k}(a)=O(n^{k-1})$.
    The paired elements therefore contribute $O(n)\cdot O(n^{k-1})=O(n^k)$. 
    The unpaired elements contribute only $O(1) \cdot O(n^k)=O(n^k)$ and the claim follows.

    We can now conclude that, for large enough even $n$,
    \[
        b_{n+2}+\ell < b_n+\ell+\eps n^{k+1},
    \]
    and thus the intervals $J_n$ and $J_{n+2}$ overlap.
    Hence, the union of the $J_n$'s (for $n$ even and large) is a single unbounded interval of integers.
    Since each $J_n$ consists of values $\SW_k(C_A^n)$,
    the theorem follows.
\end{proof}

We now prove Proposition~\ref{p:interval-n-k-plus-1}. The first step is to
construct a linear interval of switch values.

\begin{proposition}
    \label{p:linear-interval}
    For every $C>0$, there exists an integer $\ell\geq 1$ for which there are coherent collections of switches $(\fP_n^{(1)})$ such that $\fP_n^{(1)}$ is supported in
    \[
        [1,n/(3k)]\cup [n/2,3n/5],
    \]
    and
    \[
        [\ell,\ell+Cn]
        \subset
        V_{n,k}(\fP_n^{(1)}).
    \]
\end{proposition}

\begin{proof}
    By Lemma~\ref{l:even-vs-odd}, $F_{n,k}(s)$ 
    has degree $k$ and leading coefficient $-2/k!$ as a polynomial in $s$.
    Hence for every $k$-fold switch with base $s$ and steps $h_1,\ldots,h_k$, we have
    \[
        \tilde{v}_{n,k}(s;\mathbf{h})
        =
        (-1)^k \Delta_{h_1}\dots\Delta_{h_k}F_{n,k}(s)
        =
        -2h_1\dots h_k < 0,
    \]
    and thus its value is equal to $2h_1\dots h_k$.

    Let $\cP_1$ and $\cP_2$ be $k$-fold switches with steps $1,2,2^2,\dots,2^{k-1}$ and $R,3R,3^2R,\dots,3^{k-1}R$,
    respectively, where $R\geq 3$ is an odd integer to be chosen later. 
    Denote their values with $v_{n,k}(\cP_1) = 2p_1$ and $v_{n,k}(\cP_2)=2p_2$, where
    \[
        p_1=2^{k(k-1)/2},
        \qquad
        p_2=R^k3^{k(k-1)/2}.
    \]

    Take $p_2$ disjoint translates of $\cP_1$, and
    \[
        m_2
        =
        \Bigl\lfloor
            \frac{n}{10k3^kR}
        \Bigr\rfloor
    \]
    further disjoint translates of $\cP_2$, all supported inside $[1,n/(3k)]$.

    The values obtained from these switches include
    \[
        2(i_1p_1+i_2p_2),
        \qquad
        0\leq i_1\leq p_2,
        \ 
        0\leq i_2\leq m_2.
    \]
    Since $R$ is odd, $p_1$ and $p_2$ are coprime, 
    and hence elementary number theory shows that the above values include every even integer in $[2(p_2-1)p_1,\, 2m_2p_2]$.

    It remains to obtain odd values as well.
    Let $\cQ$ be the $(k-1)$-fold switch with steps $1,3,3^2,\dots,3^{k-2}$
    and base $s=n/2$.
    Note that its support lies in $[n/2,3n/5]$ for all sufficiently large even $n$.
    A direct computation gives
    \begin{equation}
        \label{eq:fin-dif-switch-val-3}
        \Delta_{h_1}\dots\Delta_{h_{k-1}}F_{n,k}(n/2)
        =
        -(h_1+\dots+h_{k-1})\, h_1\dots h_{k-1},
    \end{equation}
    and hence $\tilde{v}_{n,k}(n/2;\mathbf{h}) > 0$.
    Furthermore, $v_{n,k}(\cQ)=q$,
    where $q$ is an odd constant depending only on $k$.

    Combining $\cQ$ with the previously considered switches, 
    we obtain all integers in an interval of length at least
    \[
        2m_2p_2-2(p_2-1)p_1-q
        \asymp_k 
        R^{k-1}n.
    \]
    Thus, we may choose the odd integer $R$ so large that this length is at least $Cn$ for all sufficiently large $n$.

    The switches may be chosen coherently in $n$.
    Indeed, when $n$ is replaced by $n+2$, the number of translates changes by $O(1)$, 
    and corresponding bases change by $O(1)$. 
    The preceding sign computations show that, within each family, the same choice---interchanging the positive and negative sets or leaving them unchanged---applies for all sufficiently large $n$.
    Hence corresponding switches are translates of one another as ordered pairs.
\end{proof}

The next proposition amplifies the linear interval to an interval of length
$\Omega(n^k)$.

\begin{proposition}
    \label{p:interval-n-k}
    There exist constants $\eps_0>0$ and $\ell_0\geq 1$ for which there are coherent
    collections of switches $(\fP_n^{(k)})$ such that $\fP_n^{(k)}$ is supported in
    \[
        [1,n/4]\cup [n/2,3n/5],
    \]
    and
    \[
        [\ell_0,\ell_0+\eps_0n^k]
        \subset
        V_{n,k}(\fP_n^{(k)}).
    \]
\end{proposition}

\begin{proof}
    We prove the following stronger induction statement for scales smaller than $n^k$.
    For $1\leq r\leq k-1$ and every $C>0$, 
    there exists an integer $\ell=\ell(r,C)\geq 1$ for which
    there are coherent collections of switches $(\fP_n^{(r)})$ such that $\fP_n^{(r)}$ is supported in $[1,n/(k-r+4))\cup [n/2,3n/5]$,
    and
    \[
        [\ell,\ell+Cn^r]
        \subset
        V_{n,k}(\fP_n^{(r)}).
    \]
    The case $r=1$ is Proposition~\ref{p:linear-interval}.

    Suppose the statement holds for some $r$, where $1\leq r\leq k-1$.
    Put $q=k-r$.
    We will use $q$-fold switches to pass from scale $n^r$ to scale $n^{r+1}$.

    For fixed steps $h_1,\dots,h_q$, uniformly for $0\leq s\leq n$,
    \[
        \Delta_{h_1}\dots\Delta_{h_q}F_{n,k}(s)
        =
        -\frac{h_1\dots h_q}{r!}
        \bigl(
            s^r+(-1)^q (n-s)^r
        \bigr)
        +
        O_{\mathbf h}(n^{r-1}).
    \]
    Since $s\in [n/(q+4), n/(q+3))$ is bounded away from $n/2$,
    the expression $(n-s)^r+(-1)^qs^r$ is positive and of order $n^r$.
    Consequently, $\widetilde v_{n,k}(s;\mathbf h)<0$ uniformly over all these bases for all sufficiently large $n$.
    Furthermore, 
    there is a constant $c_r>1$, depending only on $k$, such that for all sufficiently large $n$,
    \[
        \frac{1}{c_r}(h_1\dots h_q)n^r
        \leq
        \bigl|
            \Delta_{h_1}\dots\Delta_{h_q}F_{n,k}(s)
        \bigr|
        \leq
        c_r(h_1\dots h_q)n^r
    \]
    for every such $s$.

    Let $\cR$ be a $q$-fold switch with steps $R, 2R, 2^2R, \dots, 2^{q-1}R$,
    where $R$ is a positive integer to be chosen.
    Its step product is $P_R \vcentcolon= R^q2^{q(q-1)/2}$.
    Choose
    \[
        m
        =
        \Bigl\lfloor 
            \frac{n}{4k^2 2^qR}
        \Bigr\rfloor
    \]
    disjoint translates of $\cR$, all supported inside $[n/(q+4), n/(q+3))$, which is possible for all sufficiently large $n$.

    Each of these switches has value between
    \[
        \frac{1}{c_r}P_Rn^r
        \qquad\text{and}\qquad
        c_rP_Rn^r.
    \]
    Choose $C_0>c_rP_R$. By the induction hypothesis applied with $C=C_0$,
    there is a compatible collection $\fP_n^{(r)}$ whose attainable
    values contain an interval of length $C_0n^r$. Since each new switch has
    value at most $C_0n^r$, adding the new switches one by one preserves a
    single interval of attainable values. Hence the combined system realizes
    an interval of length at least
    \[
        C_0n^r
        +
        m\cdot \frac{1}{c_r}P_Rn^r
        \geq
        \gamma_R n^{r+1}
    \]
    for some constant $\gamma_R\asymp_k R^{q-1}$.

    The support of the new switches lies in $[n/(q+4), n/(q+3))$,
    while the old support lies in $[1,n/(q+4))\cup[n/2,3n/5]$.
    Therefore the combined support lies in $[1,n/(q+3))\cup[n/2,3n/5]$,
    as required for the next stage. The construction is coherent in $n$:
    when $n$ is replaced by $n+2$,
    the number of translates changes by $O(1)$, paired bases change by $O(1)$, 
    and the uniform sign conclusion above ensures that, for each pair of corresponding switches, the same choice of whether to interchange the positive and negative sets is made.

    If $q\geq 2$, equivalently $r\leq k-2$, 
    then $\gamma_R$ can be made arbitrarily large by choosing $R$ large.
    Thus the induction statement holds for $r+1$ with an arbitrary prescribed leading constant.

    If $q=1$, equivalently $r=k-1$, 
    then $\gamma_R$ is bounded below by a positive constant depending only on $k$.
    This gives the conclusion of the proposition with some $\eps_0>0$.
\end{proof}

We now complete the proof of the main interval proposition.

\begin{proof}[Proof of Proposition~\ref{p:interval-n-k-plus-1}]
    Let $\eps_0>0$ and $\ell_0$ be as in Proposition~\ref{p:interval-n-k}.

    For $\lambda\in(0,1)$,
    \[
        F_{n,k}(\lambda n)
        =
        \frac{1-\lambda^k-(1-\lambda)^k}{k!}n^k
        +
        O(n^{k-1}).
    \]
    Hence, we may choose constants $\frac34<\alpha<\beta<1$ such that, 
    for all sufficiently large $n$ and all $a\in[\alpha n,\beta n]$,
    \[
        \frac{\eps_0}{4}n^k
        \leq
        F_{n,k}(a)
        \leq
        \frac{\eps_0}{2}n^k.
    \]

    Let $\fP_n^{(k)}$ be the switch collection given by
    Proposition~\ref{p:interval-n-k}, so that
    \[
        [\ell_0,\ell_0+\eps_0n^k]
        \subset
        V_{n,k}(\fP_n^{(k)}).
    \]
    For each $a\in[\alpha n,\beta n]$ add the singleton switch $(\emptyset,\{a\})$
    and note that the supports of these singleton switches are disjoint from the supports of switches in $\fP_n^{(k)}$.

    Each singleton switch has value at most $\eps_0 n^k / 2$,
    whereas the already constructed interval has length $\eps_0 n^k$. 
    Therefore adding the singleton switches one by one preserves a single interval of attainable values. 
    The total gain from the singleton switches is at least
    \[
        \bigl| [\alpha n,\beta n] \bigr|
        \cdot
        \frac{\eps_0}{4}n^k
        \gg
        n^{k+1}.
    \]
    Hence there exists $\eps>0$ such that
    \[
        [\ell_0,\ell_0+\eps n^{k+1}]
        \subset
        V_{n,k}(\fP_n)
    \]
    for coherent collections $(\fP_n)$.
\end{proof}

\section{A local central limit theorem}
\label{s:lclt}

Let $k\geq 2$ be an integer, which is considered fixed throughout this section. Let $n \in \N$ be sufficiently large in terms of $k$. In this section we establish a local limit theorem for the sum of a binomial random subset of $\{F_{n,k}(a) : a\in[n-1]\}$. Recall the definition
\[
    F_{n,k}(x) \vcentcolon= \binom{n-1}{k} - \binom{x}{k} - \binom{n-x}{k}
\]
and also
\[
    g_{n,k} \vcentcolon= \gcd\bigl\{F_{n,k}(x) : x \in \Z\bigr\}.
\]
Let $(X_j)_{j\in[n-1]}$ be independent identically distributed Bernoulli random variables with parameter $p \in (0,1)$. We form the random subset sum 
\[
    S \vcentcolon= \sum_{j=1}^{n-1}F_{n,k}(j)X_j,
\]
which naturally takes values in the lattice $g_{n,k}\Z$. Let $\mu \vcentcolon= \E[S]$ and $\sigma^2 \vcentcolon= \mathrm{Var}(S)$ be the expectation and variance of $S$ respectively. We have
\[
    \mu = p\sum_{j=1}^{n-1}F_{n,k}(j), \quad \sigma^2 = p(1-p)\sum_{j=1}^{n-1}F_{n,k}(j)^2.
\]
Note that $F_{n,k}(j)$ is maximised for $j\in \{\lfloor n/2\rfloor, \lceil n/2\rceil\}$, and $F_{n,k}(j) \asymp_k n^k$ for $n/3 \leq j \leq 2n/3$ say. In particular, it follows that
\[
    \mu \asymp_k pn^{k+1}, \quad \sigma^2 \asymp_k p(1-p)n^{2k+1}.
\]
We denote by
\[
    \varphi(x) \vcentcolon= \frac{1}{\sqrt{2\pi}}e^{-x^2/2}
\]
the probability density function of the standard normal distribution and by $\phi_Y(t) \vcentcolon= \E[e^{itY}]$ the characteristic function of a random variable $Y$.

The main result of this section is the following quantitative local central limit theorem.

\begin{theorem}
    \label{t:lclt}
    Suppose that $p \in [\frac{1}{n},\frac{1}{2}]$ and $m \in g_{n,k}\Z$. Then
    \begin{equation}
        \label{eq:lclt}
        \Bigl|\frac{\sigma}{g_{n,k}}\P(S=m) - \varphi\Bigl(\frac{m-\mu}{\sigma}\Bigr)\Bigr| \ll_k (pn)^{-1/6} + \sigma\exp\bigl(-\Omega_k(pn)\bigr).
    \end{equation}
\end{theorem}

We have the following immediate corollary.

\begin{corollary}
    \label{c:subset-sums}
    For any $k\geq 2$, there exists a constant $C_k > 0$ such that the following holds. Let $n\in\N$ be sufficiently large and let $m \in [C_kn^k\log n, \sum_{j=1}^{n-1}F_{n,k}(j)-C_kn^k\log n] \cap g_{n,k}\Z$. Then there exists a subset $A \subseteq [n-1]$ such that $\sum_{a\in A}F_{n,k}(a) = m$.
\end{corollary}
\begin{proof}
    By symmetry, we may assume without loss of generality that $m \leq \frac{1}{2}\sum_{j=1}^{n-1}F_{n,k}(j)$. We apply Theorem~\ref{t:lclt} with $p = m/\sum_{j=1}^{n-1}F_{n,k}(j) \gg_k (C_k\log n)/n$. Recalling that $\sigma \ll_k n^{k+\frac{1}{2}}$, it follows that, provided $C_k > 0$ is sufficiently large,
    \[
        \Bigl|\frac{\sigma}{g_{n,k}}\P(S = m) - \frac{1}{\sqrt{2\pi}}\Bigr| \leq o_k(1).
    \]
    Hence, if $n$ is large enough, then $\P(S = m) > 0$, so the desired conclusion follows.
\end{proof}

To prove Theorem~\ref{t:lclt}, we employ an approach based on characteristic functions, i.e.\ Fourier analysis. Specifically, by Fourier inversion, for $m \in g_{n,k}\Z$ we have
\begin{equation}
    \label{eq:fourier-inv}
    \P(S = m) = \frac{1}{2\pi}\int_{-\pi}^{\pi}\phi_{S/g_{n,k}}(t)e^{-itm/g_{n,k}}\,dt.
\end{equation}
The idea is to split the range of integration into two regions: a small interval around $0$ and the rest. The former can be thought of as a `major arc' and will account for the main term in the asymptotic. Similarly, the latter corresponds to the `minor arcs' in the context of the circle method, and will give rise to an error term.

The major arc can be dealt with using a result of Giuliano and Weber \cite{GiulianoWeber}, which in our setting reads as follows.

\begin{lemma}
    \label{l:main-term}
    Let $\delta \in (0,\frac{2}{3}]$ and $\tau \vcentcolon= \delta(p\sum_{j=1}^{n-1}F_{n,k}(j)^3)^{-1/3}$. Then for any $m \in g_{n,k}\Z$,
    \[
        \Bigl|\frac{\sigma}{2\pi g_{n,k}}\int_{-\tau g_{n,k}}^{\tau g_{n,k}}\phi_{S/g_{n,k}}(t)e^{-itm/g_{n,k}}\,dt - \varphi\Bigl(\frac{m-\mu}{\sigma}\Bigr)\Bigr| \ll \sigma\tau\delta^3 + e^{-\sigma^2\tau^2/2}.
    \]
\end{lemma}
\begin{proof}
    This is immediate from \cite[Lemma 3.4]{GiulianoWeber}, applied with $\nu = n-1$, $k_j = F_{n,k}(j)/g_{n,k}$, $\vartheta_j = 1-p$ for $j \in [\nu]$ and $\delta/(2\pi)$, $m/g_{n,k}$ in place of $\delta$ and $n$ respectively.
\end{proof}

As indicated in \cite{GiulianoWeber}, the estimation of the minor arcs is where the main difficulty lies. The bounds from \cite[\S3]{GiulianoWeber} are not sufficient for our purposes, so we establish the following.

\begin{proposition}
    \label{p:domination}
    For any $t \in [-\pi,\pi]$,
    \begin{equation*}
        \bigl|\phi_{S/g_{n,k}}(t)\bigr| \leq \exp\Bigl(-\min\bigl(\Omega_k(p(1-p)n), 2\sigma^2t^2/(\pi g_{n,k})^2\bigr)\Bigr).
    \end{equation*}
\end{proposition}

We first quickly show how Theorem~\ref{t:lclt} follows from Lemma~\ref{l:main-term} and Proposition~\ref{p:domination}.

\begin{proof}[Proof of Theorem~\ref{t:lclt} assuming Proposition~\ref{p:domination}]
    By~\eqref{eq:fourier-inv} and Lemma~\ref{l:main-term}, the left-hand side in~\eqref{eq:lclt} is at most a constant times
    \[
        \sigma\tau\delta^3 + e^{-\sigma^2\tau^2/2} + \frac{\sigma}{2\pi g_{n,k}}\int_{[-\pi,\pi]\setminus[-\tau g_{n,k},\tau g_{n,k}]}|\phi_{S/g_{n,k}}(t)|\,dt,
    \]
    where $\delta \in (0,\frac{2}{3}]$ is arbitrary and $\tau \vcentcolon= \delta(p\sum_{j=1}^{n-1}F_{n,k}(j)^3)^{-1/3}$. By Proposition~\ref{p:domination} and a change of variables, the last term in the preceding display is at most
    \[
        \frac{\sigma}{g_{n,k}}\exp\bigl(-\Omega_k(p(1-p)n)\bigr) + \frac{1}{2\pi}\int_{\R\setminus[-\sigma\tau,\sigma\tau]}e^{-2s^2/\pi^2}\,ds,
    \]
    where the integral is at most $e^{-\Omega(\sigma^2\tau^2)}$. The conclusion now follows by choosing $\delta \vcentcolon= \frac{2}{3}(pn)^{-1/12}$ and noting that, with this choice of $\delta$, we have $\sigma\tau \asymp_k (pn)^{1/12}$.
\end{proof}

It remains to prove Proposition~\ref{p:domination}, and we devote the remainder of the section to this task. We start with the following result, which states that polynomial images of dense subsets of intervals expand quickly under addition. This fact is presumably known to experts, but we were unable to locate an exact reference in the literature, so we provide a proof.

\begin{proposition}
    \label{p:iter-sumset}
    Let $d$ be a positive integer. Then there exist $c > 0$ and $q \in \N$ such that the following holds. Let $f \colon \Z \to \Z$ be a polynomial function of degree $d$. Let $N \in \N$ and $A \subseteq [N]$ be such that $|A| \geq N/2$. Then
    \[
        \Bigl|\Bigl\{\sum_{j=1}^{q}f(x_j) : x_1,\ldots,x_q \in A\Bigr\}\Bigr| \geq cN^d.
    \]
\end{proposition}

In the proof, we will use the following variant of Weyl's inequality, due to Green and Tao \cite[Lemma 4.4]{GreenTao}. In what follows, we use the standard notation $e(\theta) \vcentcolon= e^{2\pi i\theta}$ for $\theta\in\R$.

\begin{lemma}
    \label{l:weyl-ineq}
    Let $f \colon \Z \to \R$ be a polynomial of degree $d$ with leading coefficient $\alpha$. Suppose that
    \[
        \bigl|\E_{n\in[N]}e(f(n))\bigr| \geq \delta
    \]
    for some $\delta \in (0,\frac{1}{2})$. Then there exists $\ell \in \N$ such that $\ell \ll \delta^{-O_d(1)}$ and
    \[
        \lVert \ell\alpha\rVert_{\R/\Z} \ll \delta^{-O_d(1)}/N^d.
    \]
\end{lemma}

\begin{proof}[Proof of Proposition~\ref{p:iter-sumset}]
    The argument is similar to that in \cite[\S3]{Green}. Certainly $f$ has rational coefficients, so by clearing denominators, we may assume it has integer coefficients. For an integer $y$, write $r(y)$ for the number of $q$-tuples $(x_1,\ldots,x_q) \in A^q$ such that
    \[
        \sum_{j=1}^{q}f(x_j) = y.
    \]
    Then by the Cauchy--Schwarz inequality,
    \begin{equation}
        \label{eq:cauchy-schwarz}
        |A|^{2q} = \Bigl(\sum_{y\in\Z}r(y)\Bigr)^2 \leq |\supp(r)|\Bigl(\sum_{y\in\Z}r(y)^2\Bigr).
    \end{equation}
    But by a standard Fourier-analytic calculation, we have
    \begin{align}
        \sum_{y\in\Z}r(y)^2 
        &\leq
        \Bigl|\Bigl\{(x_1,\ldots,x_q,x_{q+1},\ldots,x_{2q}) \in [N]^{2q} : \sum_{j=1}^{q}f(x_j) = \sum_{j=q+1}^{2q}f(x_j)\Bigr\}\Bigr| \nonumber\\
        &= N^{2q}\int_{0}^{1}|S(\theta)|^{2q}\,d\theta, \label{eq:fourier-energy}
    \end{align}
    where for $\theta \in [0,1]$ we define the normalised exponential sum
    \[
        S(\theta) \vcentcolon= \E_{n\in[N]}e(f(n)\theta).
    \]
    By layer cake representation, we have
    \[
        \int_{0}^{1}|S(\theta)|^{2q}\,d\theta = \int_{0}^{1}2q\delta^{2q-1}\lambda(L_{\delta})\,d\delta,
    \]
    where $L_{\delta} \vcentcolon= \{\theta \in [0,1] : |S(\theta)| \geq \delta\}$ and $\lambda$ denotes the Lebesgue measure on $[0,1]$. Writing $a \in \Z\setminus\{0\}$ for the leading coefficient of $f$, Lemma \ref{l:weyl-ineq} implies that, for $\delta \in (0,\frac{1}{2})$,
    \[
        L_{\delta} \subseteq \bigcup_{1 \leq \ell \leq C\delta^{-C}}\bigl\{\theta \in [0,1] : \lVert \ell a\theta\rVert_{\R/\Z} \leq C\delta^{-C}N^{-d}\bigr\},
    \]
    where $C > 0$ is a constant depending only on $d$. Since the map $\theta \mapsto \{\ell a\theta\}$ preserves $\lambda$ (here $\{\cdot\}$ denotes the fractional part), the union bound implies
    \begin{align*}
        \lambda(L_{\delta}) 
        &\leq 
        \sum_{1\leq \ell \leq C\delta^{-C}}
            \lambda(\{\theta\in[0,1] : \lVert\theta\rVert_{\R/\Z} \leq C\delta^{-C}N^{-d}\}) \\
        &\leq
        C\delta^{-C}\cdot 2C\delta^{-C}N^{-d} 
        =
        2C^2\delta^{-2C}N^{-d}.
    \end{align*}
    For $\delta \in [\frac{1}{2},1]$ we have $L_{\delta} \subseteq L_{\delta'}$ for any $\delta' \in (0,\frac{1}{2})$, whence $\lambda(L_{\delta}) \leq 2^{2C+1}C^2N^{-d}$. Therefore, we obtain that
    \[
        \int_{0}^{1}|S(\theta)|^{2q}\,d\theta \leq 4qC^2N^{-d}\Bigl(\int_{0}^{1/2}\delta^{2q-1-2C}\,d\delta + \int_{1/2}^{1}2^{2C}\delta^{2q-1}\,d\delta\Bigr).
    \]
    Provided we choose say $q = \lceil C\rceil+1$, it follows that
    \[
        \int_{0}^{1}|S(\theta)|^{2q}\,d\theta \leq 4qC^2(1+2^{2C})N^{-d}.
    \]
    The desired conclusion now follows by combining the above with~\eqref{eq:cauchy-schwarz} and~\eqref{eq:fourier-energy}.
\end{proof}

In the proof of Proposition~\ref{p:domination}, we will use the following simple fact about recurrence on the torus, which states that if many consecutive multiples of a number are close to an integer, then this number must be very close to an integer.

\begin{lemma}
    \label{l:recurrence}
    Let $\theta \in \R$ and $L \in \N$ be such that $\lVert\ell\theta\rVert_{\R/\Z} \leq c$ for all $\ell \in [L]$, where $c < \frac{1}{3}$. Then $\lVert\theta\rVert_{\R/\Z} \leq c/L$.
\end{lemma}
\begin{proof}
    By translating $\theta$ by an integer, we may assume without loss of generality that $\theta \in [-\frac{1}{2},\frac{1}{2})$. Then we know in particular that $|\theta| \leq c$, and our task is to show that $|\theta| \leq c/L$. Assuming the opposite, we may choose the least $\ell \in [L]$ such that $|\ell\theta| > c$. By minimality, we then have
    \[
        |\ell\theta| \leq |(\ell-1)\theta| + |\theta| \leq c + c = 2c.
    \]
    Thus, if $m$ is the integer closest to $\ell\theta$, then
    \[
        |m| \leq |m-\ell\theta|+|\ell\theta| \leq c + 2c = 3c < 1,
    \]
    whence $m = 0$. But this means that $|\ell\theta| \leq c$, which gives the desired contradiction. 
\end{proof}

We are now ready to prove Proposition~\ref{p:domination}.

\begin{proof}[Proof of Proposition~\ref{p:domination}]
    Since $(X_j)_{j\in[n-1]}$ are independent, we have
    \[
        \bigl|\phi_{S/g_{n,k}}(t)\bigr| = \prod_{j=1}^{n-1}\bigl|\phi_{X_j}(F_{n,k}(j)\, t/g_{n,k})\bigr|.
    \]
    By \cite[Lemma 1]{GilmerKopparty}, for $t \in \R$ we have
    \[
        \bigl|\phi_{X_j}(t)\bigr| \leq 1 - 8p(1-p)\lVert t/(2\pi)\rVert_{\R/\Z}^2
    \]
    and hence
    \begin{equation}
        \label{eq:char-fn-bound}
        \bigl|\phi_{S/g_{n,k}}(t)\bigr| \leq \exp\Biggl(-8p(1-p)\sum_{j=1}^{n-1}\Bigl\lVert\frac{F_{n,k}(j)t}{2\pi g_{n,k}}\Bigr\rVert_{\R/\Z}^2\Biggr).
    \end{equation}
    Fix $t \in [-\pi,\pi]$ and write $\theta \vcentcolon= t/(2\pi) \in [-\frac{1}{2},\frac{1}{2}]$. Let $c_k \in (0,\frac{1}{2})$ be a sufficiently small constant, depending only on $k$. Let $B \vcentcolon= \mathrm{Bohr}(\theta;c_k)$, where for $\rho > 0$ we define the rank-$1$ Bohr set
    \[
        \mathrm{Bohr}(\theta;\rho) \vcentcolon= \{x \in \Z : \lVert x\theta\rVert_{\R/\Z} \leq \rho\}.
    \]
    On considering the set
    \[
        A \vcentcolon= \{j \in [n-1] : F_{n,k}(j)/g_{n,k} \in B\},
    \]
    we obtain the bound
    \[
        \bigl|\phi_{S/g_{n,k}}(t)\bigr| \leq \exp\bigl(-8p(1-p) \,c_k^2\, \bigl|[n-1]\setminus A \bigr|\bigr).
    \]
    Consequently, we may assume that
    \begin{equation}
        \label{eq:A-lower-bound}
        |A| 
        \geq 
        (1-c_k)(n-1),
    \end{equation} 
    as otherwise we are done. 
    
    Let $\eta \vcentcolon= 0$ if $k$ is even and $\eta \vcentcolon= 1$ if $k$ is odd. By Lemma~\ref{l:even-vs-odd}, we have $\deg F_{n,k} = k-\eta$ and $g_{n,k} \asymp_k n^{\eta}$. Applying Proposition~\ref{p:iter-sumset} to $f = F_{n,k}/g_{n,k}$, we obtain a positive integer $q \ll_k 1$ such that
    \[
        \Bigl|\Bigl\{\sum_{i=1}^{q}F_{n,k}(j_i)/g_{n,k} : j_1,\ldots,j_q \in A\Bigr\}\Bigr| \gg_k n^{k-\eta}.
    \]
    Let $D$ be the set on the left-hand side in the preceding display. Then $\{0\} \subseteq D \subseteq [0,O_k(n^{k-\eta})]$, so we may apply a result of Nathanson and S\'ark\"ozy \cite{NathansonSarkozy} as stated in \cite[Theorem 3.1]{Green} to obtain that, for some $q' \ll_k 1$, the iterated sumset $q'D$ contains an arithmetic progression of length at least $n^{k-\eta}/q'$ and common difference $d \leq q'$. By passing to a subprogression, we may assume that $d$ is divisible by $k!$. 
    
    Observe the key property of $A$ that for any $\ell \in \N$, any $\lambda_1,\ldots,\lambda_\ell \in \Z$ and $j_1,\ldots,j_\ell \in A$ we have
    \[
        \sum_{i=1}^{\ell}\lambda_iF_{n,k}(j_i)/g_{n,k} \in \mathrm{Bohr}\Bigl(\theta;c_k\sum_{i=1}^{\ell}|\lambda_i|\Bigr).
    \]
    In particular, $q'D \subseteq \mathrm{Bohr}(\theta;qq'c_k)$. Now if $c_k$ is chosen so that $c_k < \frac{1}{k+1}$ (and $n$ is large), it follows from \eqref{eq:A-lower-bound} that there is a $j$ such that $j,j+1,\ldots,j+k \in A$. It then follows from Lemma~\ref{l:even-vs-odd} that $\mathrm{Bohr}(\theta;2^kk!c_k)$ contains $k!$. Hence, $\mathrm{Bohr}(\theta;(qq'+2^kd)c_k)$ contains
    \[
        q'D + \{k!\ell : 0 \leq \ell < d/k!\},
    \]
    which in turn contains an arithmetic progression of length at least $n^{k-\eta}/q'$ and common difference $k!$. In summary, there exist positive integers $L \gg_k n^{k-\eta}$ and $q'' \ll_k 1$ such that for all $\ell \in [L]$ we have
    \[
        \lVert k!\ell\theta\rVert_{\R/\Z} \leq q''c_k.
    \]
    Provided $c_k$ is small enough, we may apply Lemma~\ref{l:recurrence} with $\theta' \vcentcolon= k!\theta$ in place of $\theta$ to obtain an integer $m \in [-k!/2,k!/2]$ such that $|\theta'-m| \leq q''c_k/L$. We will show that in fact $m = 0$. To this end, note that if $c_k$ is sufficiently small, then
    \begin{equation}
        \label{eq:closeness}
        \Bigl|\theta F_{n,k}(j)/g_{n,k}-\frac{mF_{n,k}(j)}{k!g_{n,k}}\Bigr| \leq \frac{1}{4k!}
    \end{equation}
    for all $j \in [n-1]$, and also $\lVert \theta F_{n,k}(j)/g_{n,k}\rVert_{\R/\Z} \leq 1/(4k!)$ for all $j \in A$. Thus, by the triangle inequality, for $j \in A$ we have
    \[
        \Bigl\lVert\frac{mF_{n,k}(j)}{k!g_{n,k}}\Bigr\rVert_{\R/\Z} \leq \frac{1}{2k!},
    \]
    which forces $k!$ to divide $mF_{n,k}(j)/g_{n,k}$, i.e.\ $K \vcentcolon= k!/(k!,m)$ to divide $F_{n,k}(j)/g_{n,k}$. 
    
    Suppose that $K > 1$ and consider the polynomial $\widetilde{F}_{n,k} \vcentcolon= 2k!F_{n,k}/g_{n,k}$. By Lemma~\ref{l:even-vs-odd}, $\widetilde{F}_{n,k}$ has integer coefficients. Furthermore, all values of $\widetilde{F}_{n,k}$ on $A$ are divisible by $2k!K$. However, by definition of $g_{n,k}$, not all values of $\widetilde{F}_{n,k}$ on $\Z$ are divisible by $2k!K$. Thus, $A$ misses at least one residue class modulo $2k!K$, which contradicts~\eqref{eq:A-lower-bound} if $c_k$ is small enough. We conclude that $K = 1$, i.e.\ $k!$ divides $m$. Since $|m| \leq k!/2$, we must have $m = 0$. From~\eqref{eq:closeness}, we now obtain that $|\theta F_{n,k}(j)/g_{n,k}| \leq 1/(4k!)$ for all $j \in [n-1]$. Finally, the bound~\eqref{eq:char-fn-bound} becomes
    \[
        \bigl|\phi_{S/g_{n,k}}(t)\bigr| \leq \exp\Bigl(-8p(1-p)\sum_{j=1}^{n-1}(\theta F_{n,k}(j)/g_{n,k})^2\Bigr) = \exp\bigl(-2\sigma^2t^2/(\pi g_{n,k})^2\bigr),
    \]
    as desired.
\end{proof}


\section{Proof of Theorem \ref{t:density}}
\label{s:density}

Given integers $2 \leq k \leq n$, let
\[
    \cT_{k,n} \vcentcolon= \{\SW_k(T) : T \text{ is a tree on $n$ vertices}\},
\]
so that, according to~\eqref{eq:all-sw-for-k},
\[
    \cT_k = \bigcup_{n=k}^{\infty}\cT_{k,n}.
\]
Let $T$ be a tree on $n$ vertices. By \cite[Theorem 3.3]{LiMaoGutman}, we have
\[
    G_{n,k} \leq \SW_k(T) \leq H_{n,k},
\]
where $H_{n,k} \vcentcolon= (k-1)\binom{n+1}{k+1}$. Note that $G_{n,k} = \SW_k(S_n)$ and $H_{n,k} = \SW_k(P_n)$, where $S_n, P_n$ are the star and path on $n$ vertices, which corresponds to the caterpillars $C_{\emptyset}^n$ and $C_{[n-1]}^n$ respectively. Recalling the definition~\eqref{eq:gcd} and noting that $g_{n,k}$ divides $F_{n,k}(0) = -\binom{n-1}{k-1}$, we conclude from~\eqref{eq:sw-trees} that
\[
    \cT_{k,n} \subseteq [G_{n,k},H_{n,k}] \cap g_{n,k}\N =\vcentcolon I_{n,k}^{\mathrm{hi}}.
\]
On the other hand, Corollary~\ref{c:subset-sums} provides an approximate converse by telling us that
\[
    \cT_{k,n} \supset [G_{n,k} + C_kn^k\log n, H_{n,k}-C_kn^k\log n] \cap g_{n,k}\N =\vcentcolon I_{n,k}^{\mathrm{lo}},
\]
where $C_k > 0$ is a constant depending only on $k$. In particular, as $n$ grows, asymptotically almost all multiples of $g_{n,k}$ from $[G_{n,k},H_{n,k}]$ are the Steiner--Wiener $k$ index of an $n$-vertex tree.

If $k$ is even, the above information is more than enough to deduce that $\cT_k$ is cofinite,
thus giving a new proof of the corresponding direction of Theorem \ref{t:cofinite}.
Indeed, for $n$ even, it follows from the differencing identity~\eqref{eq:fin-dif-switch-val-3} with $h_1 = \ldots = h_{k-1}=1$ that $g_{n,k}$ is odd, whence $g_{n,k} = 1$ by Lemma~\ref{l:even-vs-odd}. Thus, $I_{n,k}^{\mathrm{lo}}$ is a discrete interval, and a simple calculation reveals that for $n$ large enough, the intervals $I_{n,k}^{\mathrm{lo}}$ and $I_{n+2,k}^{\mathrm{lo}}$ overlap. Therefore, the cofiniteness of $\cT_k$ follows.

If $k$ is odd, then by Lemma~\ref{l:even-vs-odd} we have $g_{n,k} \asymp_k n$, which makes matters more complicated. Nevertheless, we have the inclusions
\[
    \cTlo_k \subseteq \cT_k \subseteq \cThi_k,
\]
where we define
\[
    \cTlo_k \vcentcolon= 
    \bigcup_{n=k}^{\infty}I_{n,k}^{\mathrm{lo}}, \qquad
    \cThi_k 
    \vcentcolon= 
    \bigcup_{n=k}^{\infty}I_{n,k}^{\mathrm{hi}}.
\]
We will show that this actually determines $\cT_k$ up to a set of density zero.
To this end, our main tool will be the following special case of Ford's result on integers with a divisor in a given interval \cite[Theorem 1]{Ford}.
For $1\leq y<z\leq x$, 
let $H(x,y,z)$ denote the number of positive integers at most $x$ having a divisor in $(y,z]$. 
If $100\leq y\leq\sqrt{x}$, $2y\leq z\leq y^2$, and $z=y^{1+u}$, then
\begin{equation}
    \label{eq:ford-basic}
    H(x,y,z)
    \asymp
    xu^\delta\log(2/u)^{-3/2},
\end{equation}
where $\delta$ is the Erd\H{o}s--Tenenbaum--Ford constant as defined in \eqref{eq:ETFconstant}.
We will mainly use this estimate through the following lemma, which states that integers with a divisor on a fixed power-scale occupy a set of density zero. For technical reasons, we also allow polylogarithmic factors in the endpoints of the interval.

\begin{lemma}
    \label{l:dyadic-int}
    Let $\gamma \in (0,\frac{1}{2})$, $\kappa_1, \kappa_2 \in \R$ and $\lambda_1, \lambda_2 > 0$ be such that $\kappa_1 \leq \kappa_2$ and $\lambda_1 < \lambda_2/2$. Then
    \[
        d\Bigl(\bigl\{N \in \N : N \text{ has a divisor in } \bigl(\lambda_1N^{\gamma}(\log_2 N)^{\kappa_1},\lambda_2 N^{\gamma}(\log_2 N)^{\kappa_2}\bigr]\bigr\}\Bigr) = 0.
    \]
\end{lemma}
\begin{proof}
    Throughout the proof, we allow implicit constants to depend on $\gamma,\kappa_1,\kappa_2,\lambda_1,\lambda_2$. Denote the set in question by $E$. By dyadic summation, for sufficiently large $M \in \N$ we have
    \[
        |E \cap [M]| 
        \leq C +  
        \sum_{J_0 \leq j < \log_2M} H\bigl(2^{j+1},2^{\gamma j}j^{\kappa_1}\lambda_1,2^{\gamma(j+1)}(j+1)^{\kappa_2}\lambda_2\bigr),
    \]
    where $C, J_0 \ll 1$. By~\eqref{eq:ford-basic} and a short calculation, the summand has order of magnitude
    \[
        \Theta\Bigl(2^j\bigl((\kappa_2-\kappa_1)\log j+1\bigr)^{\delta}j^{-\delta}(\log j)^{-3/2}\Bigr).
    \]
    Hence, given any $\eps > 0$, choosing $J \in \N$ so that this is at most $2^j\eps$ for $j \geq J$, we obtain
    \[
        |E \cap [M]| \ll \sum_{J_0 \leq j < J}2^j + \eps\sum_{J \leq j < \log_2M}2^j \ll 2^J + \eps M.
    \]
    Therefore, we have
    \[
        \limsup_{M\to\infty}\frac{|E\cap [M]|}{M} \leq \eps,
    \]
    so the claim follows since $\eps > 0$ is arbitrary.
\end{proof}

\begin{remark}
    \label{r:density}
    In the literature, one often sees the function $h$ being defined as
    \[
        h(\alpha,\beta) = \lim_{x\to\infty}\frac{H(x,x^\alpha,x^\beta)}{x}.
    \]
    Using Lemma~\ref{l:dyadic-int}, one can show that this is equivalent to our definition~\eqref{eq:localised-density} in the regime of interest, namely for $\beta < \frac{1}{2}$. Indeed, it suffices to show that, for large $x$,
    \[
        \bigl|\bigl\{n \leq x : n \text{ has a divisor in } (x^{\alpha}, x^{\beta}]\bigr\} \triangle \bigl\{n\leq x : n \text{ has a divisor in } (n^{\alpha}, n^{\beta}]\bigr\}\bigr| = o(x).
    \]
    We may restrict attention to values of $n$ from the interval $(x/\log x,x]$ say. Now if such an $n$ has a divisor in $(x^{\alpha},x^{\beta}]$, but does not have one in $(n^{\alpha},n^{\beta}]$, then it must have a divisor in $(n^{\beta},10(n\log n)^{\beta}]$. However, by Lemma~\ref{l:dyadic-int}, such numbers occupy a set of density zero. Similarly, if $n$ has a divisor in $(n^{\alpha},n^{\beta}]$, but not in $(x^{\alpha},x^{\beta}]$, then it has a divisor in $(n^{\alpha},10(n\log n)^{\alpha}]$. Again by Lemma~\ref{l:dyadic-int}, the set of all such $n$ has density $0$, so the desired equivalence follows.
\end{remark}

We are now in a position to relate the sets $\cTlo_k$, $\cThi_k$ and $\cT_k$.

\begin{proposition}
    \label{p:density-zero}
    Let $k \geq 3$ be odd. Then $d(\cThi_k\setminus\cTlo_k) = 0$, so in particular 
    \[
        d(\cT_k \triangle \cThi_k) = 0.
    \]
\end{proposition}
\begin{proof}
    Observe that
    \[
        \cThi_k\setminus\cTlo_k
        \subset
        U\cup V,
    \]
    where
    \begin{align*}
        U
        &\vcentcolon=
        \bigcup_{n=k}^{\infty}
        \bigl([G_{n,k},G_{n,k}+C_kn^k\log n)\cap g_{n,k}\N\bigr),\\
        V
        &\vcentcolon=
        \bigcup_{n=k}^{\infty}
        \bigl((H_{n,k}-C_kn^k\log n,H_{n,k}]\cap g_{n,k}\N\bigr).
    \end{align*}
    It thus suffices to show that
    \begin{equation}
        \label{eq:density-zero}
        d(U) = 0 = d(V).
    \end{equation}
    For the first equality, note that any $N \in U$ must have a divisor $d$ (corresponding to $g_{n,k}$) such that $d^k \ll_k N \ll_k d^k\log d$, that is, $(N/\log N)^{1/k} \ll_k d \ll_k N^{1/k}$. Hence, the conclusion follows from Lemma~\ref{l:dyadic-int} applied with $\gamma = 1/k$, $\kappa_1 = -1/k$, $\kappa_2 = 0$ and some constants $\lambda_1,\lambda_2$ depending only on $k$.

    The proof of the second equality in~\eqref{eq:density-zero} is similar. 
    Any $N \in V$ must have a divisor $d$ such that $d^{k+1} \asymp_k N$, i.e.\ $d \asymp_k N^{1/(k+1)}$. Hence, we are again done by Lemma~\ref{l:dyadic-int}, now applied with $\gamma = 1/(k+1)$ and $\kappa_1 = \kappa_2 = 0$.
\end{proof}

By Proposition~\ref{p:density-zero}, it suffices
to study the set $\cThi_k$, which admits a fairly clean description in terms of divisors.
We will compare it with the even simpler localised-divisor set
\[
    \cD_k 
    \vcentcolon=
    \bigl\{N \in \N : \text{$N$ has a divisor in }(N^{1/(k+1)}, N^{1/k}]\bigr\}.
\]
To carry out this comparison, we will need finer control on $g_{n,k}$. This is provided by the following lemma, whose proof is deferred to Appendix~\ref{s:gcd}.

\begin{lemma}
    \label{l:gcd}
    For any odd $k \geq 3$ and any $n \geq k$, we have
    \[
        g_{n,k} = \frac{n-k+1}{\gcd(n-k+1,\lcm(1,\ldots,k-1))}.
    \]
\end{lemma}

We can now establish a relationship between the sets $\cT_k$ and $\cD_k$.

\begin{proposition}
    \label{p:TkDk}
    Let $k\geq 3$ be odd. Then
    \[
        d(\cT_k\triangle \cD_k)=0.
    \]
\end{proposition}

\begin{proof}
    By Proposition~\ref{p:density-zero}, it suffices to prove the same statement with $\cT_k$ replaced by $\cThi_k$. In what follows, we will denote by $\alpha_k,\beta_k > 0$ a sufficiently small and sufficiently large constant depending only on $k$, respectively; their values may vary from line to line.
    
    Note that any $N \in \cThi_k$ has a divisor $d$ with $d^k \ll_k N \ll_k d^{k+1}$, i.e.\ $N^{1/(k+1)} \ll_k d \ll_k N^{1/k}$. 
    Hence, if $N \not\in \cD_k$, then $N$ has a divisor in either $(\alpha_kN^{1/(k+1)}, N^{1/(k+1)}]$ or $(N^{1/k}, \beta_kN^{1/k}]$. By Lemma \ref{l:dyadic-int}, the density of the set of all $N$ satisfying either of these conditions is $0$, whence 
    \[
        d(\cThi_k \setminus \cD_k) = 0.
    \]

    Conversely, if $N \in \cD_k\setminus \cThi_k$, then $N$ has a divisor $d \in (N^{1/(k+1)}, N^{1/k}]$. Letting
    \[
        n \vcentcolon= \lcm(1,\ldots,k-1)d+k-1,
    \]
    Lemma~\ref{l:gcd} implies that $g_{n,k} = d$. Since $N \not\in I_{n,k}^{\mathrm{hi}}$, we must have either $N < G_{n,k}$ or $N > H_{n,k}$. In particular, either $d \gg_k N^{1/k}$ or $d \ll_k N^{1/(k+1)}$. Thus, $N$ has a divisor in either $(\alpha_kN^{1/k}, N^{1/k}]$ or $(N^{1/(k+1)}, \beta_kN^{1/(k+1)}]$. Again by Lemma~\ref{l:dyadic-int}, the set of all such $N$ has density zero, so it follows that
    \[
        d(\cD_k\setminus \cThi_k) = 0,
    \]
    as desired.
\end{proof}

At last, we can give the short deduction of Theorem~\ref{t:density}.

\begin{proof}[Proof of Theorem~\ref{t:density}]
    By Proposition~\ref{p:TkDk}, the sets $\cT_k$ and $\cD_k$ differ by a set of density zero, so in particular $d(\cT_k)$ exists and equals $d(\cD_k) = h(\frac{1}{k+1},\frac{1}{k})$. That $h(\frac{1}{k+1},\frac{1}{k})$ has order of magnitude $\Theta(k^{-\delta}(\log k)^{-3/2})$ follows from~\eqref{eq:ford-basic} together with a theorem of Haddad \cite[Theorem 1]{Haddad} (alternatively, one can use a short dyadic argument similar to that in the proof of Lemma~\ref{l:dyadic-int}).
\end{proof}


\section{Concluding remarks and open problems}
\label{s:further}

In Theorem~\ref{t:cofinite} we have shown that, for even $k$, every sufficiently large integer $N$ is the Steiner--Wiener $k$-index of a tree. Letting
\[
    e(k) \vcentcolon= |\N\setminus\cT_k|, \qquad N_0(k) \vcentcolon= \max(\N\setminus\cT_k)
\]
be the number of exceptions and maximum exception respectively, the results of \cite{Wagner, WangYu} show that $e(2) = 49$ and $N_0(2) = 159$. We pose the following problem.

\begin{problem}
    \label{prob:exceptions}
    Estimate $e(k)$ and $N_0(k)$ for even $k \geq 4$.
\end{problem}

Our arguments are of an effective nature and one could extract from them a bound on $e(k)$ and $N_0(k)$. However, we have made no attempt to optimise this aspect of our arguments, so the resulting bounds are likely to be poor.

For odd $k$, Theorem~\ref{t:density} identifies the density of $\cT_k$ with the classical localised-divisor density $h(\frac{1}{k+1},\frac{1}{k})$.
Although Theorem~\ref{t:density} gives the order of magnitude for this density, it falls short of giving more precise estimates.
This leads to the following problem.

\begin{problem}
    \label{prob:asymptotic}
    Find more precise estimates for the localised-divisor density $h(\frac{1}{k+1},\frac{1}{k})$, and in particular an asymptotic equivalent as $k \to \infty$.
\end{problem}

From a computational perspective, the very recent work of Drappeau and Mounier \cite{drappeau2026computingsieveintegralsusing} is directly relevant to this problem.
They give effective finite-integral formulae for $h(\alpha,\beta)$, together with rigorous numerical methods based on integration over polytopes. In principle, their methods apply to $h(\frac{1}{k+1},\frac{1}{k})$ and hence to the density of $\cT_k$.
However, the present family lies in a thin localised-divisor regime: the gap $\frac1k-\frac1{k+1}$ shrinks quadratically with $k$.
This is precisely the direction in which the complexity of the finite-integral formulae grows quickly. 
Thus the family arising here provides a natural test case for extending effective computations of localised-divisor densities closer to the diagonal $\alpha=\beta$.

In a slightly different direction, there are well-studied analogies between the prime factorisation of a typical integer and the cycle structure of random permutations (see \cite{FordAnatomy, GranvilleAnatomy}). From this perspective, a natural model for Problem~\ref{prob:asymptotic} might be to investigate the limiting probability (as $n \to \infty$) that a random permutation from $S_n$ has an invariant set of size between $n/(k+1)$ and $n/k$. In the regime when the size of the invariant set is fixed and equal to $k$, an asymptotic for this quantity has very recently been established by Green and Sawhney \cite{GreenSawhneyPermutations}. This work improves on the results of Eberhard, Ford and Green \cite{EberhardFordGreen}, which established an order of magnitude of $k^{-\delta}(\log k)^{-3/2}$ for this quantity and was in turn based on arguments of Ford \cite{Ford}. In view of these connections and in particular the similar numerology, the methods of  \cite{GreenSawhneyPermutations} and the forthcoming work \cite{GreenSawhneyMultTable} might be relevant to the problem at hand.

{\small
\subsection*{Acknowledgements and AI tool disclosure}
The first author would like to thank Christian Bernert for useful discussions.

The authors are supported by the Croatian Science Foundation under the project no.\ HRZZ-IP-2022-10-5116 (FANAP) and by the European Union – NextGenerationEU through the National Recovery and Resilience Plan 2021-2026 Institutional grant of University of Zagreb Faculty of Science (IK IA 1.1.3.\ Impact4Math).

The statement and proof of Lemma~\ref{l:gcd} were suggested by ChatGPT Plus 5.5. 
In addition, the same tool was used to assist with 
routine typesetting,
grammar checking,
phrasing,
identifying technical corrections such as tracking constants, 
and locating many of the references. 
Apart from these uses, the text of this paper was human-written.
}

\appendix

\section{Proof of Lemma~\ref{l:gcd}}
\label{s:gcd}

In this section we prove Lemma~\ref{l:gcd}, whose statement we repeat here for convenience.

\begin{lemma}
    \label{l:gcd-app}
    For any odd $k \geq 3$ and any $n \geq k$, we have
    \[
        g_{n,k} = \frac{n-k+1}{\gcd(n-k+1,\lcm(1,\ldots,k-1))}.
    \]
\end{lemma}

In the proof, we will require the following elementary number-theoretic fact.

\begin{lemma}
    \label{l:binomial-gcd}
    For any positive integers $m,\ell$, we have
    \[
        \gcd\Bigl\{\binom{m+j-1}{j} : 1 \leq j \leq \ell\Bigr\} = \frac{m}{\gcd(m,\lcm(1,\ldots,\ell))}.
    \]
\end{lemma}
\begin{proof}
    It suffices to show that, for any prime $p$,
    \[
        \min_{j\in [\ell]}v_p\Bigl(\binom{m+j-1}{j}\Bigr) = v_p(m) - \min(v_p(m), \gamma),
    \]
    where $v_p$ denotes $p$-adic valuation and we write
    \[
        \gamma \vcentcolon= v_p\bigl(\lcm(1,\ldots,\ell)\bigr) = \max_{j\in [\ell]}v_p(j) = \lfloor\log_p\ell\rfloor.
    \]
    By Kummer's theorem (see e.g.\ \cite[Theorem 3.7]{Granville}), the left-hand side is the smallest number of carries when adding some $j \in [\ell]$ to $m-1$ in base $p$. If $v_p(m) \leq \gamma$, then the number of carries for $j = p^{v_p(m)}$ is zero, so the claim follows. If $v_p(m) > \gamma$, then it is easy to see the number of carries is always at least $v_p(m)-\gamma$, with equality for $j = p^{\gamma}$. Thus, the proof is complete.
\end{proof}

\begin{proof}[Proof of Lemma~\ref{l:gcd-app}]
    We recall the standard fact that a polynomial with rational coefficients maps $\Z$ to $\Z$ if and only if all of its coefficients with respect to the binomial basis
    \[
        \cB \vcentcolon= \Bigl\{\binom{x}{j} : j \geq 0\Bigr\}
    \]
    are integers. Furthermore, for such a polynomial $f$,
    \[
        \gcd\bigl\{f(x) : x \in \Z\bigr\}
    \]
    can be computed as the greatest common divisor of its coefficients with respect to $\cB$. Expanding $F_{n,k}$ in terms of $\cB$, we obtain
    \[
        F_{n,k}(x) = -\binom{n-1}{k-1} -\sum_{j=1}^{k-1}(-1)^j\binom{n-j}{k-j}\binom{x}{j}.
    \]
    Therefore, we have
    \[
        g_{n,k} = \gcd\Bigl\{\binom{n-j}{k-j} : 1 \leq j \leq k-1\Bigr\},
    \]
    so the conclusion follows from Lemma~\ref{l:binomial-gcd}.
\end{proof}

\begingroup
\sloppy
\printbibliography
\endgroup

\end{document}